\documentclass{article}
\usepackage{graphicx}
\usepackage{amsmath, amsthm}

\theoremstyle{plain}
\newtheorem{theorem}{Theorem}[section]

\newtheorem{corollary}[theorem]{Corollary}

\theoremstyle{definition}

\numberwithin{equation}{section}
\date{}

\begin{document}
\title{Determination of time-like helices from intrinsic equations in Minkowski 3-Space}

\author{Ahmad T. Ali\\
Mathematics Department\\
 Faculty of Science, Al-Azhar University\\
 Nasr City, 11448, Cairo, Egypt\\
E-mail: \textit{atali71@yahoo.com}\\
\vspace*{0.5cm}\\
Melih Turgut \footnote{Corresponding author.}\\
Department of Mathematics, \\
Buca Educational Faculty, Dokuz Eyl\"{u}l University,\\
35160 Buca, Izmir, Turkey\\
E-mail: \textit{melih.turgut@gmail.com}}

\maketitle
\begin{abstract}
In this paper, position vectors of a time-like curve with respect to standard frame of Minkowski space E$^3_1$ are studied in terms of Frenet equations. First, we prove that position vector of every time-like space curve in Minkowski space E$^3_1$ satisfies a vector differential equation of fourth order. The general solution of mentioned vector differential equation has not yet been found. By special cases, we determine the parametric representation of the general helices from the intrinsic equations (i.e. curvature and torsion are functions of arc-length) of the time-like curve. Moreover, we give some examples to illustrate how to find the position vector from the intrinsic equations of general helices.\\

\textbf{M.S.C. 2000}: 53C40, 53C50\\
\textbf{Keywords}: Classical differential geometry; Frenet equations;  general helix; Intrinsic equations.
\end{abstract}

\section{Introduction}
In the local differential geometry, we think of curves as a geometric set of points, or locus. Intuitively, we are thinking of a curve as the path traced out by a particle moving in E$^3$. So, investigating position vectors of the curves is a classical aim to determine behavior of the particle (or the curve, i.e.). There exists a vast literature on this subject, for instance \cite{chen}, \cite{Il1}, \cite{Il2}, \cite{Il3}, \cite{Tur}. The aim of these works is to obtain position vectors of the curves with respect to Frenet frame. And, in the classical differential geometry, it is well-known that determining position vector of an arbitrary curve according to standard frame is not easy. In a recent study, \cite{Yil} obtain position vectors of space-like $W-$curves according to standard frame of E$_1^3$ by means of vector differential equations.

A curve of constant slope or general helix is defined by the property that the tangent lines make a constant angle with a fixed direction. A necessary and sufficient condition that a curve to be general helix is that ratio of curvature to torsion be constant. Indeed, a helix is a special case of the gerenal helix. If both curvature and torsion are non-zero constants, it is called a helix or only a $W-$curve.

Helices arise in nanosprings, carbon nanotubes, $\alpha-$helices, DNA double and collagen triple helix, the double helix shape is commonly associated with DNA, since the double helix is structure of DNA, \cite{cci}. This fact was published for the first time by Watson and Crick in 1952 (see \cite{Wc}). They constructed a molecular model of DNA in which there were two complementary, antiparallel (side-by-side in opposite directions) strands of the bases guanine, adenine, thymine and cytosine, covalently linked through phosphodiesterase bonds (for details, see \cite{cci}, \cite{Cgj}, \cite{cook}).

All helices ($W-$curves) in E$^3_1$ are completely classified by Walfare in \cite{walfare}. For instance, the only planar space-like degenerate helices are circles and hyperbolas. In \cite{Il3}, the authors investigated position vectors of a time-like and a null helix ($W-$curve) with respect to Frenet frame.

In this work, we use vector differential equations established by means of Frenet equations in Minkowski space E$^3_1$ to determine position vectors of the time-like curves according to standard frame of E$^3_1$. We obtain position vectors of a time-like general helix with respect to standard frame of E$^3_1$. We hope these results will be helpful to mathematicians who are specialized on mathematical modeling.

\section{Preliminaries}
To meet the requirements in the next sections, here, the basic elements
of the theory of curves in the space $E_{1}^{3}$ are briefly presented (A
more complete elementary treatment can be found in \cite{Onei}.)\\

The Minkowski 3-space $E_{1}^{3}$ is the Euclidean 3-space $E^{3}$
provided with the standard flat metric given by
\begin{center}
$g =-dx_{1}^{2}+dx_{2}^{2}+dx_{3}^{2},$
\end{center}
where $(x_{1},x_{2},x_{3})$ is a rectangular coordinate system of $E_{1}^{3}$
.
Since $g$ is an indefinite metric, recall that a vector $v\in E_{1}^{3}$ can have one of three Lorentzian characters: it can be space-like if $g(v,v)>0$ or $v=0$, time-like if $g(v,v)<0$ and null if $g(v,v)=0$ and $v\neq 0$. Similarly, an arbitrary curve $\varphi =\varphi (s)$ in $E_{1}^{3}$ can locally be space-like, time-like or null (light-like), if all of its velocity vectors $\varphi ^{\prime }$ are respectively space-like, time-like or null (light-like), for every $s\in I\subset R$. The pseudo-norm of an arbitrary vector $a\in E_{1}^{3}$ is given by $\left\Vert a\right\Vert=\sqrt{\left\vert g(a,a) \right\vert }$. $\varphi $
is called an unit speed curve if\ velocity vector $v$ of $\varphi $
satisfies $\left\Vert v\right\Vert =1$. For vectors $v,w\in $ $E_{1}^{3}$ it is said to be orthogonal if and only if $g(v,w) =0$.\\

Denote by $\left\{ T,N,B\right\} $ the moving Frenet frame along the curve $\varphi $ in the space $E_{1}^{3}$. For an arbitrary curve $\varphi $ with first and second curvature, $\kappa $ and $\tau $ in the space $E_{1}^{3}$, the following Frenet formulae are given in \cite{Il3}:\\

If $\varphi $ is a time-like curve, then the Frenet formulae read
\begin{eqnarray}\label{u2}
\left[
\begin{array}{c}
T^{\prime } \\
N^{\prime } \\
B^{\prime }%
\end{array}%
\right] =\left[
\begin{tabular}{lll}
$0$ & $\kappa $ & $0$ \\
$\kappa $ & $0$ & $\tau $ \\
$0$ & $-\tau $ & $0$%
\end{tabular}%
\right] \left[
\begin{array}{c}
T \\
N \\
B%
\end{array}%
\right],
\end{eqnarray}
where
\begin{center}
$g(T,T)=-1$, $g(N,N)=g(B,B)=1$,\\
$g(T,N)=g(T,B)=g(T,N)=g(N,B)=0.$
\end{center}

Recall that an arbitrary curve is called a $W-$curve if it has constant Frenet curvatures \cite{Il2}. And, from the view of Differential Geometry, a helix is a geometric curve with non-vanishing constant curvature $\kappa$ and non-vanishing constant torsion $\tau$ \cite{Il3}.

\section{Main Results for Time-like Curves}
In this section, first, we adapt important theorems in the classical differential geometry of the curves to time-like curves of Minkowski 3-space.\\

Lipschutz \cite{lips} stated and proved the following two important theorem in Euclidean space E$^3$. Here we state the same theorems in Minkowski space E$_1^3$ but without proof.

\begin{theorem}\label{th-main} A curve is defined uniquely by its curvature and torsion as function of a natural parameters.
\end{theorem}
The equations
$$
\kappa=\kappa(s),\,\,\,\,\,\tau=\tau(s)
$$
which give the curvature and torsion of a curve as functions of $s$ are called the natural or intrinsic equations of a curve, for they completely define the curve.

We observe that the Frenet equations form a system of three vector differential equations of the first order in $T, N$ and $B$. It is reasonable to ask, therefore, given arbitrary continuous functions $\kappa$ and $\tau$, whether or not there exist solutions $T,N,B$ of the Frenet equations, and hence, since $\psi^{\prime}=T$, a curve
$$
\psi=\int Tds+C
$$
which the prescribed curvature and torsion. The answer is in the affirmative and is given by

\begin{theorem}{\bf (Fundamental existence and uniqueness theorem for space curve).}\label{th-main} Let $\kappa(s)$ $\tau(s)$ be arbitrary continuous function on $a\leq s \leq b$. Then there exists, except for position in space, one and only one timelike curve $C$ for which $\kappa(s)$ is the curvature, $\tau(s)$ is the torsion and $s$ is a natural parameter along $C$.
\end{theorem}

The problem of the determination of parametric representation of the position vector of an arbitrary space curve according to the intrinsic equations is still open in the Euclidean space E$^3$ and in the Minkowski spae $E_1^3$ \cite{eisenh, lips}. This problem is not easy to solve in general case. We solved this problem in the case of the general helix ($\dfrac{\tau}{\kappa}$ is constant) in Minkowski space E$_1^3$.\\

In the light of above statements, first, we give:
\begin{theorem}\label{th-main} Let $\psi=\psi(s)$ be a time-like unit speed curve. Then, position $\psi$ satisfies a vector differential forth order as follows
\begin{equation}\label{u21}
\dfrac{d}{ds}\Big[\dfrac{1}{\tau}\dfrac{d}{ds}\Big(\dfrac{1}{\kappa}\dfrac{d^2\psi}{ds^2}\Big)\Big]+
\Big(\dfrac{\tau}{\kappa}-\dfrac{\kappa}{\tau}\Big)\dfrac{d^2\psi}{ds^2}-
\dfrac{d}{ds}\Big(\dfrac{\kappa}{\tau}\Big)\dfrac{d\psi}{ds}=0.
\end{equation}
\end{theorem}

\begin{proof}
Let $\psi =\psi (s)$ be an unit speed time-like curve with nonvanishing curvature and torsion. If we substitute first equation of (\ref{u2}) to second equation of (\ref{u2}), we have
\begin{equation}\label{u3}
B=\frac{d}{ds}\left( \frac{1}{\kappa }\frac{dT}{ds}\right)- \frac{\kappa }{\tau }T.
\end{equation}
Differentiating of (\ref{u3}) and using in third equation of (\ref{u2}), we write
\begin{equation}
\frac{d}{ds}\left[ \frac{1}{\tau }\frac{d}{ds}\left( \frac{1}{\kappa }\frac{dT}{ds}%
\right) \right] +\left(\frac{\tau }{\kappa }- \frac{\kappa }{\tau }
\right) \frac{dT}{ds}- \frac{d}{ds}\left( \frac{\kappa }{\tau }
\right) T=0.
\end{equation}
Denoting $\dfrac{d\psi }{ds}=T$, we have the following vector differential
equation of fourth order
\begin{equation}\label{u4}
\frac{d}{ds}\left[ \frac{1}{\tau }\frac{d}{ds}\left( \frac{1}{\kappa }\frac{d^{2}\psi }{ds^{2}}\right) \right] +\left(\frac{\tau }{\kappa }- \frac{\kappa }{\tau }\right) \frac{d^{2}\psi }{ds^{2}}- \frac{d}{ds}\left( \frac{
\kappa }{\tau }\right) \frac{d\psi }{ds}=0.
\end{equation}
\end{proof}

If we put $\tau(s)=\dfrac{\kappa(s)}{f(s)}$, the equation (\ref{u4}) takes the following simple form
\begin{equation}\label{u5}
\frac{d}{d\theta}(f\,\frac{d^2T}{d\theta^2})+(\frac{1-f^2}{f})\frac{dT}{d\theta}-
\frac{df}{d\theta}T=0,\,\,\,\,f=f(\theta),\,\,\,\,\theta=\int\kappa(s)ds.
\end{equation}

By means of solution of the above equation, position vector of an arbitrary space curve can be determined. However, the general solution of it has not been found. So, we investigate special cases.
\begin{theorem}\label{th-main} The position vector of a time-like general helix can be computed in the natural parameter form
\begin{equation}\label{u211}
\psi(s)=\sinh[\alpha]\int\Big(
\coth[\alpha],\cos\Big[\int\mathrm{csch}[\alpha]\kappa(s)ds\Big],
\sin\Big[\int\mathrm{csch}[\alpha]\kappa(s)ds\Big]\Big)ds+C
\end{equation}
or in the parametric form
\begin{equation}\label{u212}
\psi(\phi)=\int\dfrac{\sinh^2[\alpha]}{\kappa(\phi)}\Big(\coth[\alpha],\cos[\phi], \sin[\phi]\Big)d\phi+C,\,\,\,\phi=\mathrm{csch}[\alpha]\int\kappa(s)ds.
\end{equation}
\end{theorem}

\begin{proof}
If $\psi$ is a general helix whose tangent vector $\psi^{\prime}$ makes a constant angle $\alpha$ with a constant direction $U$, then we can write $f(\theta)=a$, where $a$ is constant value. Therefore the equation (\ref{u212}) becomes
\begin{equation}\label{u6}
\dfrac{d^3T}{d\theta^3}+(\dfrac{1}{a^2}-1)\,\dfrac{dT}{d\theta}=0.
\end{equation}
Here, there exists two cases:

{\bf Case 3.1:} $\dfrac{1}{a^2}-1\geq0$.

{\bf Case 3.2:} $\dfrac{1}{a^2}-1\leq0$.

If we change the variable $\theta$, the equation (\ref{u6}) can be written in the following form:
\begin{equation}\label{u61}
\dfrac{d^3T}{d\phi_\pm^3}\pm\dfrac{dT}{d\phi_\pm}=0, \,\,\,\,\,\phi\pm=\sqrt{\pm(\dfrac{1}{a^2}-1)}\,\,\theta.
\end{equation}
If we write the tangent vector $T=\Big(T_1, T_2, T_3\Big)$ the general solution of Eq. (\ref{u61}) takes the form
\begin{equation}\label{u7}
T(\phi\pm)=T_i(\phi\pm)e_i=
\left\{\begin{array}{ll}
&(a_i\cosh[\phi_-]+b_i\sinh[\phi_-]+c_i)e_i,\,\,\,\dfrac{1}{a^2}-1\leq0,\\
&(a_i\cos[\phi_+]+b_i\sin[\phi_+]+c_i)e_i, \,\,\,\dfrac{1}{a^2}-1\geq0,
\end{array}
\right.
,\,i=1,2,3,
\end{equation}
where $a_i, b_i, c_i\in R$ for i=1,2,3.\\

Hence the curve $\psi$ is a general helix, i.e. the tangent vector $T$ makes an constant angle $\alpha$ with the constant vector called the axis of the helix. So, with out loss of generality, we can take the axis of helix is parallel to one of the axis $x$-axis $e_1$, $y$-axis $e_2$ or $z$-axis $e_3$. Because $T$ is a timelike vector, the vector parallel to axis must be timelike vector. It is worth noting that: the only timelike vector parallel to $x$-axis. If we write $e_1=(-1,0,0)$, we have $c_1=T_1=g(T,e_1)=\cosh[\alpha]$ and $a_3=b_3=0$. The constant angle is called the hyperbolic angle between two timelike vector $T$ and $e_1$.

On other hand the tangent vector $T$ is a unit timelike vector, so ,we have the following condition
\begin{equation}\label{u8}
-T_1^2+T_2^2+T_3^2=-1,
\end{equation}
which leads to
\begin{equation}\label{u9}
\begin{array}{ll}
T_2^2+T_3^2=\sinh^2[\alpha].
\end{array}
\end{equation}
Here, we study two cases as the following:\\

{\bf Case 3.1.1}
\begin{equation}\label{u919}
\begin{array}{ll}
T=&\cosh[\alpha]e_1+\Big(a_2\cosh[\phi_+]+b_2\sinh[\phi_+]+c_2\Big)e_2\\
&+\Big(a_3\cosh[\phi_+]+b_3\sinh[\phi_+]+c_3\Big)e_3.
\end{array}
\end{equation}
The above equation can be written in the form
\begin{equation}\label{u109}
A_0+\sum_{j=1}^2\Big(A_j\cosh[j\phi_+]+B_i\sinh[j\phi_+]\Big)=0,
\end{equation}
where
\begin{equation}\label{u119}
\left\{\begin{array}{ll}
A_2&=\dfrac{1}{2}\Big(a_2^2+a_3^2+b_2^2+b_3^2\Big)\\
B_2&=a_2b_2+a_3b_3\\
A_1&=2(a_2c_2+a_3c_3)\\
B_1&=2(b_2c_2+b_3c_3)\\
A_0&=c_2^2+c_3^3-\sinh^2[\alpha]+\dfrac{1}{2}\Big[a_2^2+a_3^2-b_2^2-b_3^2\Big].
\end{array}\right.
\end{equation}
All coefficients of Eq. (\ref{u109}) must be zero, but $A_2$ is the sum of positive values, then there are no solutions in this case.\\

{\bf Case 3.2.1}
\begin{equation}\label{u91}
\begin{array}{ll}
T=&\cosh[\alpha]e_1+\Big(a_2\cos[\phi_+]+b_2\sin[\phi_+]+c_2\Big)e_2\\
&+\Big(a_3\cos[\phi_+]+b_3\sin[\phi_+]+c_3\Big)e_3.
\end{array}
\end{equation}

The above equation can be written in the form
\begin{equation}\label{u10}
A_0+\sum_{j=1}^2\Big(A_j\cos[j\phi_+]+B_i\sin[j\phi_+]\Big)=0,
\end{equation}
where
\begin{equation}\label{u11}
\left\{\begin{array}{ll}
A_2&=\dfrac{1}{2}\Big(a_2^2+a_3^2-b_2^2-b_3^2\Big)\\
B_2&=a_2b_2+a_3b_3\\
A_1&=2(a_2c_2-a_3c_3)\\
B_1&=2(b_2c_2+b_3c_3)\\
A_0&=c_2^2+c_3^3-\sinh^2[\alpha]+\dfrac{1}{2}\Big[a_2^2+a_3^2+b_2^2+b_3^2\Big].
\end{array}\right.
\end{equation}

All coefficients of Eq. (\ref{u10}) must be zero, so we have the following set of five algebraic equations in the six unknowns $a_2, b_2, c_2, a_3, b_3$ and $c_3$.
\begin{equation}\label{u12}
A_i=0,\,\,\, \forall\,\,\, i=1,2,...,5.
\end{equation}
Solving the five algebraic equations above, with the aid of mathematica programm, we obtain four cases of solutions as the following:
\begin{equation}\label{u13}
\left\{\begin{array}{ll}
c_2=c_3=0,\,\,\,\,b_3=a_2,\,\,\,\, a_3=-b_2=\pm\sqrt{\sinh^2[\alpha]-a_2^2}\\
c_2=c_3=0,\,\,\,\,b_3=-a_2,\,\,\,\, a_3=b_2=\pm\sqrt{\sinh^2[\alpha]-a_2^2}.
\end{array}\right.
\end{equation}
The four cases above lead to one general form solution and equation (\ref{u7}) takes the following form:
\begin{equation}\label{u14}
\begin{array}{ll}
T(\phi)=&\cosh[\alpha]e_1+\Big(a_2\cos[\phi_+]+\sqrt{\sinh^2[\alpha]-a_2^2}\sin[\phi_+]\Big)e_2\\
&+
\Big(\sqrt{\sinh^2[\alpha]-a_2^2}\cos[\phi]-a_2\sin[\phi]\Big)e_3.
\end{array}
\end{equation}

The above equation can be written in the following form:
\begin{equation}\label{u15}
\begin{array}{ll}
T(\phi)=\Big(\cosh[\alpha],\sinh[\alpha]\cos[\phi_+-\varepsilon],
\sinh[\alpha]\sin[\phi_+-\varepsilon]\Big).
\end{array}
\end{equation}
where $\varepsilon=\arctan^{-1}\Big[\sqrt{\dfrac{\sinh^2[\alpha]}{a_2^2}-1}\Big]$. Without loss of generality we can write:
\begin{equation}\label{u16}
\begin{array}{ll}
T(\phi)=\Big(\cosh[\alpha],\sinh[\alpha]\cos[\phi_+],
\sinh[\alpha]\sin[\phi_+]\Big).
\end{array}
\end{equation}
By differentiation the above equation with respect to $s$, we have
\begin{equation}\label{u162}
\begin{array}{ll}
N(\phi)=\Big(0,-\sin[\phi_+],\cos[\phi_+]\Big),
\end{array}
\end{equation}
\begin{equation}\label{u163}
\begin{array}{ll}
\dfrac{d\phi}{ds}={\mathrm{csch}}[\alpha]\kappa(s).
\end{array}
\end{equation}

The binormal vector $B$ takes the form:
\begin{equation}\label{u164}
\begin{array}{ll}
B(\phi)=-\Big(\sinh[\alpha],\cosh[\alpha]\cos[\phi_+],\cosh[\alpha]\sin[\phi_+]\Big).
\end{array}
\end{equation}

By differentiation the above equation with respect to $s$, we have
\begin{equation}\label{u165}
\begin{array}{ll}
\dfrac{d\phi}{ds}={\text{sech}}[\alpha]\tau(s).
\end{array}
\end{equation}
Here the constant $a$ must be equal $a=\dfrac{\kappa}{\tau}=\coth[\alpha]$ and so $\phi_+(s)=\sqrt{\dfrac{1}{a^2}-1}\int\kappa(s)ds={\text{csch}}[\alpha]\int\kappa(s)ds$.\\

Now, integrating the equation (\ref{u16}) with respect to $s$ we have the two equations (\ref{u211}) and (\ref{u212}). Thus, we complete proof of the theorem.
\end{proof}
\begin{corollary}
There are not a time-like general helix with $\Big|\dfrac{\tau}{\kappa}\Big|<1$.
\end{corollary}

\section{Examples}
In this section, we take several choices for the curvature $\kappa$ and torsion $\tau$, and next apply theorem 3.4.\\

{\bf Example 1.} The case of a curve when both of the curvature and torsion are constants (W-curve), i.e., $\kappa=a\,\sinh[\alpha],\,\tau=a\,\cosh[\alpha]$. Then position vector takes the form:
\begin{equation}\label{u18}
\begin{array}{ll}
\psi(s)=\sinh[\alpha]\int\Big(\coth[\alpha],
\cos[a\,s], \sin[a\,s]\Big)ds+C.
\end{array}
\end{equation}

Integrating the above equation and putting $s=\dfrac{\phi}{a}$, we obtain the parametric representation of the W-curve as the following:
\begin{equation}\label{u19}
\begin{array}{ll}
\psi(\theta)=\dfrac{\sinh[\alpha]}{a}\Big(\coth[\alpha]\,\theta,\sin[\theta], -\cos[\theta]\Big)+C.
\end{array}
\end{equation}

{\bf Example 2.} The case of a general helix with $\kappa=\dfrac{\sinh[\alpha]}{s}$ and $\tau=\dfrac{\cosh[\alpha]}{s}$. Then position vector takes the form:
\begin{equation}\label{u20}
\begin{array}{ll}
\psi=\sinh[\alpha]\int\exp[\phi]\Big(\coth[\alpha], \cos[\phi], \sin[\phi]\Big)d\phi+C,
\end{array}
\end{equation}
where $s=\exp[\phi]$. Integrating the above equation, we obtain the parametric representation of this curve as the following:

\begin{equation}\label{u21}
\begin{array}{ll}
\psi=\dfrac{\sinh[\alpha]}{2}\exp[\phi]\Big(2\coth[\alpha], \sin[\phi]+\cos[\phi], \sin[\phi]-\cos[\phi]\Big)+C.
\end{array}
\end{equation}

{\bf Example 3.} The case of a general helix with $\kappa=\dfrac{\sinh[\alpha]}{s^2+1}$ and $\tau=\dfrac{\cosh[\alpha]}{s^2+1}$. Then position vector takes the form:
\begin{equation}\label{u201}
\begin{array}{ll}
\psi=\sinh[\alpha]\int\exp[\phi]\Big(\coth[\alpha]\sec^2[\phi], \sec[\phi], \sec[\phi]\tan[\phi]\Big)d\phi+C,
\end{array}
\end{equation}
where $s=\tan[\phi]$. Integrating the above equation, we obtain the parametric representation of this curve as the following:
\begin{equation}\label{u212}
\begin{array}{ll}
\psi=\sinh[\alpha]\Big(\coth[\alpha]\tan[\phi], \ln\Big[\sec[\phi]+\tan[\phi]\Big], \sec[\phi]\Big)+C.
\end{array}
\end{equation}

\begin{figure}[htt]
\begin{center}
\includegraphics[width=4.5cm]{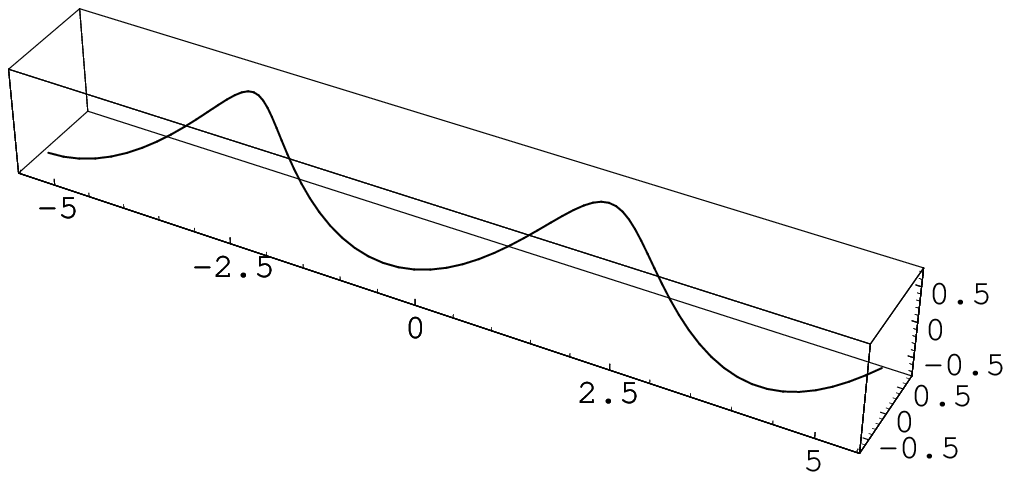}
\hspace*{0.5cm}
\includegraphics[width=3cm]{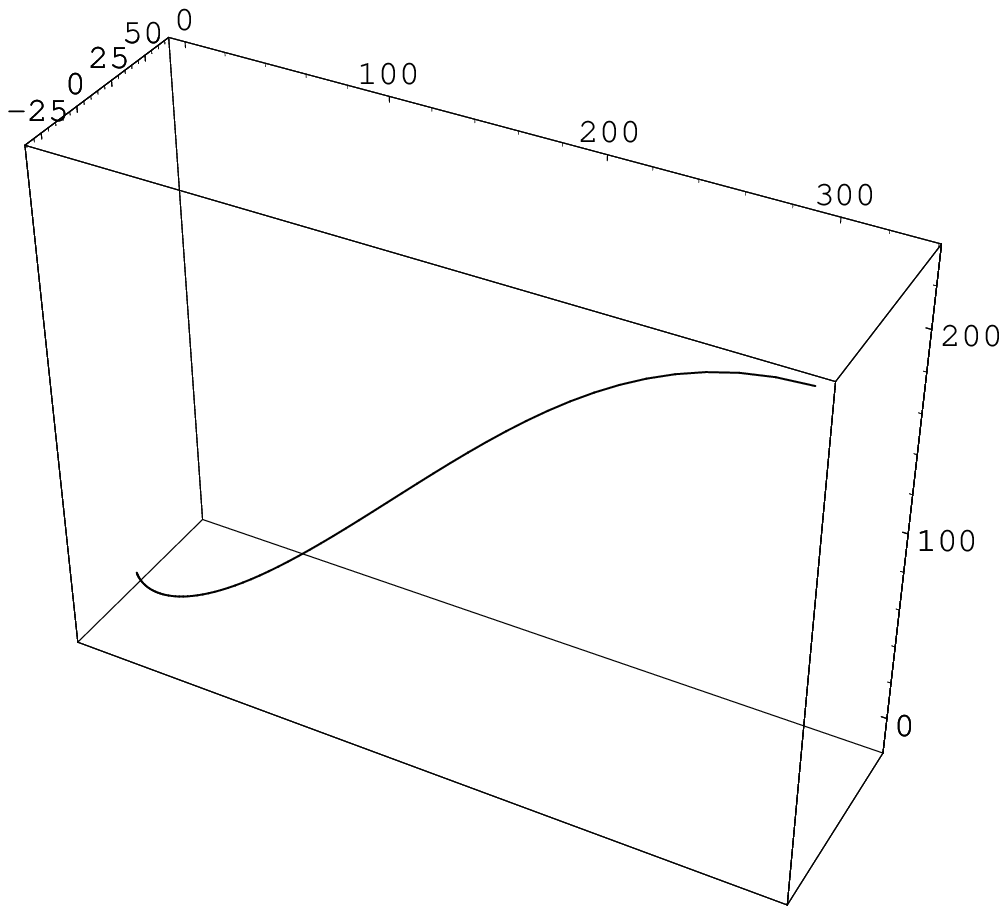}
\hspace*{0.5cm}
\includegraphics[width=3cm]{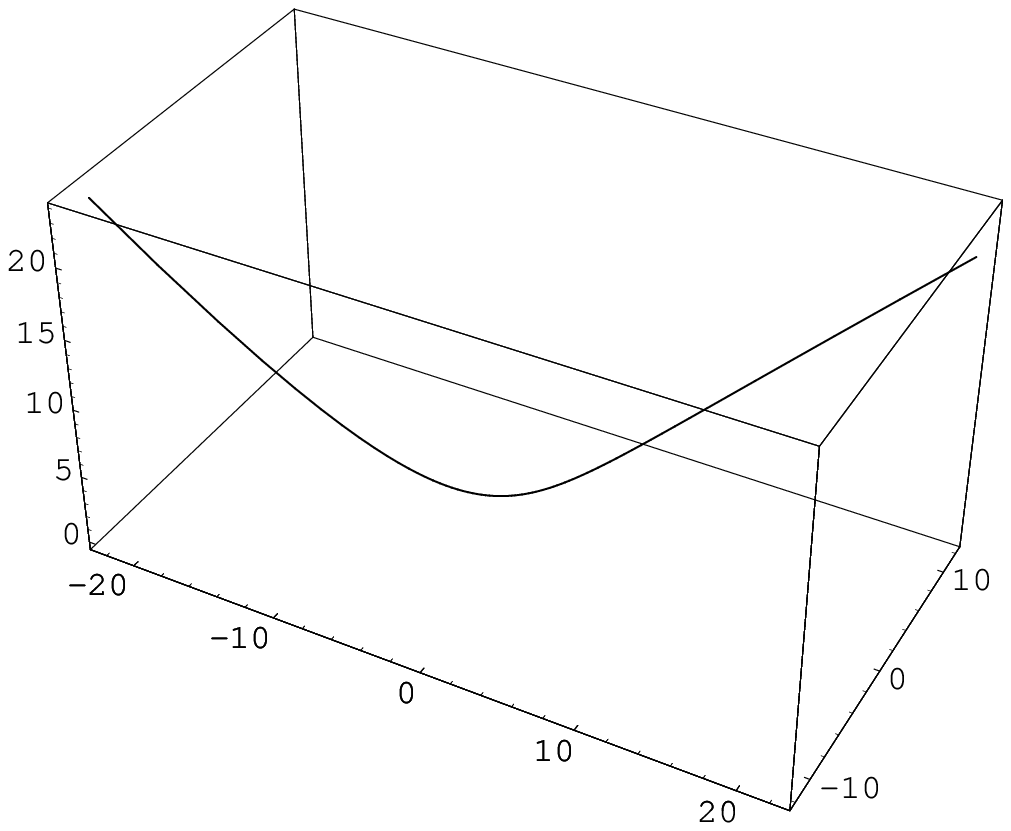}
\end{center}\caption{Time-like helices corresponding to examples 1,2 and 3.}
\label{fig-x}
\end{figure}


\begin{thebibliography}{99}
\bibitem{cci} Camc\i, \c{C}., \.{I}larslan, K., Kula, L.  and Hac\i saliho\u{g}lu, H. H. Harmonic curvatures and generalized helices in $E^4$, \textit{Chaos, Solitons
and Fractals}, 2007, doi:10.1016/j.chaos.2007.11.001 \textit{in press.}

\bibitem{chen} Chen, B. Y. When does the position vector of a space curve always lie in its rectifying plane?, \textit{Amer. Math. Mounthly}, 2003, \textbf{110}, 147-152.

\bibitem{Cgj} Chouaieb, N., Goriely, A. and Maddocks, J.H. Helices \textit{PNAS}, 2006, \textbf{103}, 398–403.

\bibitem{cook} Cook T. A. \textit{The curves of life}, Constable, London – 1914; Reprinted (Dover, London – 1979).

\bibitem{eisenh} Eisenhart LP, A Treatise on the Differential Geometry of Curves and Surfaces, Ginn and Co., 1909.

\bibitem{Il1} \.{I}larslan, K. Space-like Normal Curves in Minkowski Space E$_{1}^{3}$, \textit{Turk J. Math.}, 2005, \textbf{29}, 53-63.

\bibitem{Il2} \.{I}larslan, K. and Boyac\i o\u{g}lu, \"{O}. Position Vectors of a Space-like $W$-curve in Minkowski Space E$_1^3$, \textit{Bull. Korean Math. Soc.}  2007, \textbf{44}, 429-438.

\bibitem{Il3} \.{I}larslan, K. and Boyac\i o\u{g}lu, \"{O}. Position vectors of a timelike and a null helix in Minkowski 3-space, \textit{Chaos, Solitons and Fractals}, 2008, \textbf{38}, 1383-1389.

\bibitem{lips} Lipschutz MM. Schum$^,$s Outline of Theory and Problems of Differential Geometry. McGraw-Hill Book Company, New York, 1969.

\bibitem{Onei} O'Neill, B. {\it Semi-Riemannian Geometry}, Academic Press, New York, 1983.

\bibitem{Tur} Turgut, M.  and Y\i lmaz, S. Contributions to Classical Differential Geometry of the Curves in E$^3$, \textit{Sci. Magna}, 2008, \textbf{4}, 5-9.

\bibitem{Yil} Y\i lmaz, S. Determination of Space-like curves by Vector Differential Equations in Minkowski space E$_1^3$, J. Adv. Res. Pure Math. vol. 1 no. 1 pp. 10-14, 2009.

\bibitem{walfare} Walfare J. Curves and Surfaces in Minkowski Space. PhD thesis, K.U. Leuven, Faculty of Science, Leuven, 1995.

\bibitem{Wc} Watson J. D. and Crick F. H. Molecular structures of nucleic acids. \textit{Nature}, 1953, \textbf{171}, 737–8.

\end{thebibliography}
\end{document}